\documentstyle[amsfonts, txfonts, amssymb, epsfig, latexsym, 12pt]{amsart}

\newtheorem{theorem}{Theorem}[section]
\newtheorem{lemma}[theorem]{Lemma}
\newtheorem{proposition}[theorem]{Proposition}
\newtheorem{corollary}[theorem]{Corollary}

\def\Z{{\mbox{\rm\kern.25em
\vrule width.03em height0.57ex depth0ex
\kern.033em
\vrule width.03em height1.52ex depth-0.96ex \kern-.338em Z}}}
\def\R{{\mbox{\rm I\kern-.22em R}}}
\def\N{{\mbox{\rm I\kern-.22em N}}}

\def\max{{\rm max}}
\def\sgn{{\rm sgn}}
\def\ker{{\rm Ker}}
\def\M{{\cal{M}}}

\def\C{{\cal{C}}}
\def\M{{\cal{M}}}

\def\E{{\cal{E}}}

\def\dist{{\rm dist}}

\def\111{\gamma}

\def\be#1{\begin{equation}\label{#1}}
\def\bas{\begin{align*}}
\def\eas{\end{align*}}
\def\bi{\begin{itemize}}
\def\ei{\end{itemize}}
\newenvironment{proof}{\noindent {\bf Proof} }{\endprf\par}
\def \endprf{\hfill  {\vrule height6pt width6pt depth0pt}\medskip}
\def\emph#1{{\it #1}}

\title[]{Some remarks on the $n$-linear Hilbert transform for $n\geq 4$}

\author{Camil Muscalu}
\address{Department of Mathematics, Cornell University, Ithaca, NY 14853}
\email{camil@@math.cornell.edu}

\begin{document}

\begin{abstract}
We prove that for every integer $n\geq 4$, the $n$-linear operator whose symbol is given by a product of two generic symbols of
$n$-linear Hilbert transform type, does not satisfy any $L^p$ estimates similar to those in H\"{o}lder inequality. Then, we extend this result to multilinear operators
whose symbols are given by a product of an arbitrary number of generic symbols of $n$-linear Hilbert transform kind. 
As a consequence, under the same assumption $n\geq 4$, these immediately imply that for any $1<p_1, ..., p_n \leq \infty$ and $0<p<\infty$ with $1/p_1+ ... +1/p_n = 1/p$,
there exist non-degenerate subspaces $\Gamma \subseteq \R^n$ of maximal dimension $n-1$, and Mikhlin symbols $m$ singular along $\Gamma$, for
which the associated  $n$-linear multiplier operators $T_m$ do not map $L^{p_1}\times ... \times L^{p_n}$ into $L^p$.
These counterexamples are in sharp contrast 
with the bilinear case, where similar operators are known to satisfy many  $L^p$ estimates of H\"{o}lder type.

\end{abstract}

\maketitle

\section{Introduction}

Let $n\geq 2$. For $\vec{\alpha} = (\alpha_1, ..., \alpha_{n-1}) \in \R^{n-1}$ an arbitrary vector, consider the expression

\begin{equation}\label{1}
\int_{\R^n}
\sgn (\xi + \alpha_1 \xi_1 + ... + \alpha_{n-1}\xi_{n-1})
\widehat{f}(\xi)
\widehat{f_1}(\xi_1) ...
\widehat{f_{n-1}}(\xi_{n-1})
e^{2\pi i x(\xi+\xi_1+ ... + \xi_{n-1})}
d \xi d\xi_1 ... d \xi_{n-1}
\end{equation}
where $f, f_1, ..., f_{n-1}$ are all Schwartz functions on the real line and $x$ is a real number. If all the entries $(\alpha_j)_j$ are different than $0$ and $1$ and also different from each other,
the $n$-linear operator from \eqref{1} is called {\it the $n$-linear Hilbert transform} and it will be denoted by $n HT_{\vec{\alpha}}(f, f_1, ..., f_{n-1})(x)$.
Notice that if one erases the symbol 

\begin{equation}\label{symbol}
\sgn (\xi + \alpha_1 \xi_1 + ... + \alpha_{n-1}\xi_{n-1})
\end{equation}
from \eqref{1}, the corresponding expression becomes the product of the functions involved $$f(x) f_1(x) ... f_{n-1}(x).$$
The main question about these operators is whether they satisfy estimates of H\"{o}lder type, more precisely if there exist $1<p_1, ..., p_n \leq \infty$ and $0<p<\infty$
with $1/p_1+ ... +1/p_n = 1/p$ so that $n HT_{\vec{\alpha}}$ can be naturally extended as a bounded $n$-linear operator from $L^{p_1}\times ... \times L^{p_n}$ into $L^p$.

The interest in their study comes from their close connection to the so called Calder\'{o}n commutators \cite{calderon}, \cite{cm}. Indeed, a direct calculation shows that modulo a universal constant,
one has the identity

\begin{equation}\label{2}
\int_{[0,1]^{n-1}}
n HT_{\vec{\alpha}}(f, a, ..., a)(x) d\vec{\alpha} = p.v. 
\int_{\R}\frac{(A(x) - A(y))^{n-1} }{(x-y)^n } f(y) d y
\end{equation}
with $A'=a$. As one can recognize, the expression on the right hand side is precisely the $(n-1)$th commutator of Calder\'{o}n \cite{calderon}, \cite{cm}.
If $n=2$ and $\alpha$ is an arbitrary real number different than $0$ and $1$, {\it the bilinear Hilbert transform} $2 HT_{\alpha}$ does satisfy such estimates, thanks to the work of Lacey and Thiele from \cite{lt1} and \cite{lt2}.
But for $n\geq 3$ no positive results are presently known.

There are two very natural ways to generalize these operators, which will be described next.
First, for any $k\geq 1$ and arbitrary vectors $\vec{\alpha_1}, ..., \vec{\alpha_k}\in \R^{n-1}$ denote by $n HT_{\vec{\alpha_1}, ..., \vec{\alpha_k} }$ the $n$-linear operator defined by the {\it product symbol}

\begin{equation}\label{psymbol}
\prod_{j=1}^k \sgn (\xi + \alpha_{j\, 1}\xi_1 + ... + \alpha_{j\, n-1}\xi_{n-1})
\end{equation}
where $\vec{\alpha_j} : = (\alpha_{j\, 1}, ..., \alpha_{j\, n-1})$ for $j=1, ..., k$. The following theorem holds.

\begin{theorem}
For any $k\geq 1$ and generic numbers $\alpha_1, ..., \alpha_k$, the bilinear operator $2 HT_{\alpha_1, ..., \alpha_k}$ satisfies
many $L^p$ estimates of H\"{o}lder type \footnote{In the bilinear case, the vectors $\vec{\alpha_j}$ being one dimensional, can be identified with the numbers $\alpha_j$.}.
\end{theorem}
It is a very simple exercise to show that this theorem follows from the $k=1$ case studied in \cite{lt1} and \cite{lt2} \footnote{Indeed, if $n=2$, one can first observe that the symbol \eqref{psymbol} is constant on
various {\it angular regions} centered at the origin, and so it is enough to understand bilinear operators whose symbols are the characteristic functions of such {\it angular sets}. 
But (modulo some natural compositions with certain Riesz projections) the study of these  
can be easily reduced to the study of bilinear operators of type $2 HT_{\alpha}$. The details are left to the reader.}.
See also \cite{gn} for some related results. In other words, for bilinear operators, the $k\geq 1$ situation is as complex as the original $k=1$ case.

One can then naturally ask if there is an $n$-linear generalization of the above result. This may of course seem hopeless, given the remarks made above. 
However, we will prove in this paper, that if $k\geq 2$ and $n\geq 4$, the most natural $n$-linear generalization of the above bilinear theorem, is false. More precisely, we will prove

\begin{theorem}\label{teorema}
Let $k\geq 2$ and $n\geq 4$. Then, for any generic vectors $\vec{\alpha_1}, ..., \vec{\alpha_k}\in\R^{n-1}$, the $n$-linear operator $n HT_{\vec{\alpha_1}, ..., \vec{\alpha_k} }$ 
does not satisfy {\it any} $L^p$ estimates of H\"{o}lder type.
\end{theorem}
This time (and in fact for any $n\geq 3$), {\it the geometry of the symbols} \eqref{psymbol} becomes more complicated and there do not seem to be any direct connections between the $k=1$ and $k\geq 2$ cases. In particular, the case $k=1$ and $n\geq 3$ remains open
\footnote{ However, in this particular case, it has been noticed (first heuristically in \cite{camil1} and then rigorously in \cite{d}) that the trilinear Hilbert transform cannot map
$L^{p_1}\times L^{p_2}\times L^{p_3}$ into $L^p$ for {\it every} $1/3 < p < \infty$. See also \cite{christ} for another interesting {\it tri-linear counterexample}. }.

This brings us naturally to the second class of extensions, that we mentioned earlier. We need to set up some notations first.
For any $\vec{\beta}\in \R^{n-1}$, denote by $\Gamma_{\vec{\beta}}$ the $(n-1)$ dimensional  subspace defined by

$$\Gamma_{\vec{\beta}} := \{ (\xi, \xi_1, ..., \xi_{n-1})\in \R^n : \xi + \beta_1 \xi_1 + ... + \beta_{n-1} \xi_{n-1} = 0 \}.$$
Notice that the symbol \eqref{symbol} of $nHT_{\vec{\alpha}}$ is singular along $\Gamma_{\vec{\alpha}}$ while the symbol \eqref{psymbol} of $n HT_{\vec{\alpha_1}, ..., \vec{\alpha_k} }$ is singular 
along $\Gamma_{\vec{\alpha_1}} \cup ... \cup \Gamma_{\vec{\alpha_k}}$. Then, for any $\Gamma\subseteq \R^n$ subspace of arbitrary dimension, denote by $\M_{\Gamma}(\R^n)$ the class of
Marcinkiewicz-H\"{o}rmander-Mikhlin symbols which are singular along $\Gamma$. More specifically, $\M_{\Gamma}(\R^n)$ contains all bounded functions $m(\vec{\eta})$, which are smooth in the complement of $\Gamma$,
and which satisfy

$$|\partial^{\gamma}m (\vec{\eta})| \lesssim \frac{1}{\dist (\vec{\eta}, \Gamma)^{|\gamma|}}$$
for sufficiently many multi-indices $\gamma$. Observe also that if $\widetilde{\Gamma}\subseteq \Gamma$ one has the inclusion $\M_{\widetilde{\Gamma}}(\R^n)\subseteq \M_{\Gamma}(\R^n)$.
Given any $m\in \M_{\Gamma}(\R^n)$ one then denotes by $T_m$ the $n$-linear operator defined by the same formula \eqref{1}, with the symbol \eqref{symbol} replaced by $m$.

The following general theorem has been proved in \cite{mtt4}. See also \cite{gn} for the particular bilinear case.

\begin{theorem}\label{general}
For any non-degenerate subspace $\Gamma\subseteq \R^n$ and $m\in \M_{\Gamma}(\R^n)$, the multi-linear operator $T_m$ satisfies many $L^p$ estimates of H\"{o}lder type, as long as $\dim (\Gamma) < \frac{n+1}{2}$.
\end{theorem}
Notice that the $2HT_{\alpha}$ case corresponds to $n=2$ and $\dim (\Gamma) = 1$, while $nHT_{\vec{\alpha}}$ would be covered by the case of subspaces $\Gamma$ of maximal dimension $\dim(\Gamma) = n-1$
\footnote{The non-degeneracy of $\Gamma$ is understood in the sense of \cite{mtt4}. Without being too specific, let us just say that generic subspaces {\it are} non-degenerate. For instance, if $n=2$, any line is non-degenerate,
if it is not one of the coordinate axes, nor the one defined by the equation $\xi+\xi_1 =0$.}.

Given all these results described so far, it is natural (and quite tempting) to believe, that if $nHT_{\vec{\alpha}}$ were to satisfy some range of $L^p$ estimates, then these estimates should remain valid for generic vectors $\vec{\alpha}$,
and that at least some of them, should be available for operators of type $T_m$ corresponding to generic symbols $m$ in the class $\M_{\Gamma_{\vec{\alpha}}}(\R^n)$, as well. In other words, that there exists a non-trivial range
of exponents, where Theorem \ref{general} can be extended all the way to the maximal dimension $\dim(\Gamma) = n-1$.

However, it is not difficult to see as a consequence of the previous Theorem \ref{teorema}, that this {\it ideal scenario} cannot be true. We have

\begin{theorem}\label{mult}
For any $1<p_1, ..., p_n \leq \infty$ and $0<p<\infty$ with $1/p_1+ ... +1/p_n = 1/p$,
there exist non-degenerate subspaces $\Gamma \subseteq \R^n$ of maximal dimension $n-1$, and symbols $m \in \M_{\Gamma}(\R^n)$, for
which the associated  $n$-linear multiplier operators $T_m$ do not map $L^{p_1}\times ... \times L^{p_n}$ into $L^p$.
\end{theorem}
To prove Theorem \ref{mult} one just has to observe that any $nHT_{\vec{\alpha_1}, \vec{\alpha_2}}$ splits quite naturally as

$$nHT_{\vec{\alpha_1}, \vec{\alpha_2}} = T_{m_1} + T_{m_2}$$
where $m_1\in \M_{\Gamma_{\vec{\alpha_1}}}(\R^n)$ and $m_2\in \M_{\Gamma_{\vec{\alpha_2}}}(\R^n)$. Since for generic vectors $\vec{\alpha_1}$ and $\vec{\alpha_2}$ the $n$-linear operator $nHT_{\vec{\alpha_1}, \vec{\alpha_2}}$
does not satisfy any $L^p$ estimates (cf. Theorem \ref{teorema}), it is clearly impossible for both $T_{m_1}$ and $T_{m_2}$ to map $L^{p_1}\times ... \times L^{p_n}$ into $L^p$.

Coming now back to Theorem \ref{teorema}, its proof will follow from the following weaker, but also more precise, result.

\begin{proposition}\label{precis}
For any $k\geq 2$ there exists a positive integer $N(k)$ such that for any $n\geq N(k)$ and generic vectors $\vec{\alpha_1}, ..., \vec{\alpha_k}\in\R^{n-1}$, the $n$-linear operator $n HT_{\vec{\alpha_1}, ..., \vec{\alpha_k} }$ 
does not satisfy {\it any} $L^p$ estimates of H\"{o}lder type. Moreover, the explicit counterexamples that will be constructed, are {\rm irreducible} in a certain natural sense.
\end{proposition}
The preciseness of the proposition lies on this {\it irreducibility} property of the counterexamples. We will see later on, that given any such irreducible counterexample for $n HT_{\vec{\alpha_1}, ..., \vec{\alpha_k} }$,
it can be naturally localized, rescaled and translated in frequency, so that it automatically becomes a counterexample (which we call {\it reducible} this time) to the boundedness of any $n HT_{\vec{\alpha_1}, ..., \vec{\alpha_{k'}} }$ as long as $k'\geq k$ 
\footnote{ We thank Christoph Thiele for pointing out to us this observation.}. The constants $N(k)$ above will be quite explicit as we will prove the proposition for $N(k) = \frac{ (2k)!}{(k!)^2} - 1$.
Notice that when $k=2$, the expression $\frac{ (2k)!}{(k!)^2} - 1$ is equal to $5$, but we will remark later on that Proposition \ref{precis} remains valid even for $n=4$. In particular, Theorem \ref{teorema} will follow
from the case $k=2$ and $n\geq 4$ since as we mentioned, the {\it irreducible} counterexamples for $nHT_{\vec{\alpha_1}, \vec{\alpha_2}}$ can be transformed into {\it reducible} ones for any $n HT_{\vec{\alpha_1}, ..., \vec{\alpha_k} }$
when $k\geq 2$. Let us also remark that when both $k$ and $n$ are sufficiently large and $n\geq N(k)$, one obtains quite a few distinct classes of counterexamples for the generic $n HT_{\vec{\alpha_1}, ..., \vec{\alpha_k} }$, since besides the
irreducible ones that will be constructed explicitly, there will be various reducible counterexamples coming from the operators having a lower complexity.

It is also interesting to compare all of these negative results
with the positive ones in \cite{mtt1}, \cite{mtt2} and \cite{mtt3}.

The rest of the paper is essentially devoted to the proof of Proposition \ref{precis}. The method we use is a generalization of the arguments from \cite{mtt} and \cite{mptt}. See also \cite{fefferman}, for some somewhat related ideas.

{\bf Acknowledgement:} We wish to thank Christoph Thiele for various comments on a preliminary draft of the manuscript. The present work has been partially supported by the NSF.

\section{Some heuristical arguments}\label{s2}

Before starting the actual proof, we would like to describe a {\it heuristical proof} of Proposition \ref{precis} which will motivate the rigorous argument that will be presented afterwards.

First of all, let us observe that $n HT_{\vec{\alpha_1}, ..., \vec{\alpha_k} }(f, f_1, ..., f_{n-1})(x)$ admits the alternative {\it kernel representation}

\begin{equation}\label{kernel}
\frac{(-1)^k}{(i \pi)^k} \cdot p.v. \int_{\R^k} f(x+t_1 + ... + t_k)
\prod_{j=1}^{n-1} f_j(x + \alpha_{1\,j}t_1 + ... + \alpha_{k\,j}t_k) 
\frac{d t_1}{t_1} ... \frac{d t_k}{t_k}.
\end{equation}
This is a simple consequence of the well known identity $ i \pi \,\sgn(\xi) = \widehat{\frac{1}{t}}(\xi)$ applied $k$ times to the symbol \eqref{psymbol}. Consider now 
$f(x) = e^{ i \# x^k}$ and
$f_j(x) = e^{ i \#_j x^k}$ for $j=1, ..., n-1$ where $\#, \#_1, ..., \#_{n-1}$ are real numbers that will be determined later on.

If one plugs in these functions into the formula \eqref{kernel}, one formally obtains 

$$p.v. \int_{\R^k} 
e^{ i ( \#(x+t_1+ ... +t_k)^k + \sum_{j=1}^{n-1} \#_j (x+\alpha_{1\,j}t_1 + ... + \alpha_{k\,j}t_k)^k )}\frac{d t_1}{t_1} ... \frac{d t_k}{t_k} .
$$
The new expression 

\begin{equation}\label{polinom}
\#(x+t_1+ ... +t_k)^k + \sum_{j=1}^{n-1} \#_j (x+\alpha_{1\,j}t_1 + ... + \alpha_{k\,j}t_k)^k
\end{equation}
should be interpreted as a polynomial in the $k+1$ variables $x, t_1, ..., t_k$ which is homogeneous of degree $k$. An elementary combinatorial computation shows that
this expression has precisely $\frac{ (2k)!}{(k!)^2}$ monomials. For reasons that will be clearer a bit later, we would like to choose our numbers $\#, \#_1, ..., \#_{n-1}$
in such a way that all the coefficients of these monomials are zero with the exception of the ones corresponding to $x^k$ and $t_1\cdot ... \cdot t_k$. Let us have a look at the coefficient of $t_1^k$ for instance. It is given by

$$\# + \sum_{j=1}^{n-1} \#_j \alpha_{1\,j}^k$$
and so the fact that it is zero is equivalent to the fact that the $n$-dimensional vector $(\#, \#_1, ..., \#_{n-1})$ is orthogonal to $(1, \alpha_{1\,1}^k, ..., \alpha_{1\,n-1}^k)$. Since one can argue in a similar way for all the other monomials,
our wish becomes equivalent to the fact that $(\#, \#_1, ..., \#_{n-1})$ is orthogonal to $\frac{ (2k)!}{(k!)^2}-2$ other vectors in $\R^n$. Since $\vec{\alpha_1}, ..., \vec{\alpha_k}$ are generic, all these vectors will be
linearly independent and the fact that such a vector exists is guaranteed by the condition $n\geq \frac{ (2k)!}{(k!)^2} - 1$ stated in Proposition \ref{precis}.
Furthermore, by a proper dilation, one can also assume that the coefficient of $t_1\cdot ... \cdot t_k$ will be equal to $1$.

In particular, for any such a vector $(\#, \#_1, ..., \#_{n-1})$ one can write

\begin{equation}\label{alta}
|   n HT_{\vec{\alpha_1}, ..., \vec{\alpha_k} }   (f, f_1, ..., f_{n-1})(x)| = \frac{1}{\pi^k} \cdot  | \int_{\R^k}\frac{e^{ i  t_1\cdot ...\cdot t_k}}{t_1\cdot ...\cdot t_k} d t_1 ... dt_k |.
\end{equation}
On the other hand, the right hand side of \eqref{alta} can be further calculated as

$$
| \int_{\R^k}\frac{e^{ i  t_1\cdot ...\cdot t_k}}{t_1\cdot ...\cdot t_k} d t_1 ... dt_k | = \pi \int_{\R^{k-1}} \frac{\sgn (t_1\cdot ... \cdot t_{k-1})}{t_1 \cdot ... \cdot t_{k-1}} d t_1 ... dt_{k-1} 
= \pi\, (\int_{\R}\frac{1}{|t|} d t )^{k-1}
$$
and this means that formally, we have obtained the identity

\begin{equation}\label{alta1}
|  n HT_{\vec{\alpha_1}, ..., \vec{\alpha_k} } (f, f_1, ..., f_{n-1})(x)| = \frac{1}{\pi^{k-1}} \cdot  |f(x) f_1(x) ... f_{n-1}(x)|\cdot (\int_{\R}\frac{1}{|t|} d t )^{k-1}.
\end{equation}
Notice that while the moduli of the initial functions are all equal to $1$, the right hand side of \eqref{alta1} is infinite.
The idea now is to restrict all the functions above to an interval of type $[-N, N]$ and to observe that as long as $x$ belongs to an interval of the same size, 
one has 

\begin{equation}\label{alta2}
|    n HT_{\vec{\alpha_1}, ..., \vec{\alpha_k} }    (f, f_1, ..., f_{n-1})(x)| \geq  c (\log N)^{k-1}
\end{equation}
as $N \rightarrow \infty$. Clearly, \eqref{alta2} would imply Proposition \ref{precis} and from now on the goal is to describe a rigorous proof it
\footnote{The functions $f, f_1, ..., f_{n-1}$ which appeared in \eqref{alta2} are the old ones restricted smoothly to an interval of type $[-N, N]$.}.

\section{Proof of Proposition \ref{precis}}\label{s3}

Fix $\vec{\alpha_1}, ..., \vec{\alpha_k}$ generic vectors in $\R^{n-1}$. For $N$ large enough, define the function $\chi_N(x)$ to be the characteristic
function of the interval $[-N, N]$ and $\widetilde{\chi}_N(x)$ to be a smooth function supported on $[-N-\epsilon, N+\epsilon]$ and equal to $1$ on $[-N, N]$, where $\epsilon > 0$
is a number much smaller than $1/N^{k-1}$.

Consider real numbers $\#, \#_1, ..., \#_{n-1}$ chosen to satisfy all the requirements of Section \ref{s2}. Define the functions $f, f_1, ..., f_{n-1}$ by

$$f(x) = \widetilde{\chi}_N(x) e^{i\# x^k}$$
and 

$$f_j(x) = \widetilde{\chi}_N(x) e^{i\#_j x^k}$$
for $1\leq j\leq n-1$. We claim that there exist small constants $c$ and $\widetilde{c}$ depending on all these parameters with the exception of $N$ so that

\begin{equation}\label{1*}
|    n HT_{\vec{\alpha_1}, ..., \vec{\alpha_k} }    (f, f_1, ..., f_{n-1})(x)| \geq  c (\log N)^{k-1}
\end{equation}
as long as $x\in [-\widetilde{c} N, \widetilde{c} N]$. Clearly, as we pointed out earlier, \eqref{1*} would immediately imply Proposition \ref{precis}, since it holds for arbitrarily large $N$.

To prove the claim let us first observe that since $f, f_1, ..., f_{n-1}$ are smooth and compactly supported, formula \eqref{kernel} can be applied and one has 

$$
n HT_{\vec{\alpha_1}, ..., \vec{\alpha_k} }    (f, f_1, ..., f_{n-1})(x) =
$$

\begin{equation}\label{2*}
e^{i (\# + \#_1 + ... + \#_{n-1})x^k}
\frac{(-1)^k}{(i \pi)^k} \int_{\R^k}  \widetilde{\chi}_N  (x+t_1 + ... + t_k)
\prod_{j=1}^{n-1}  \widetilde{\chi}_N   (x + \alpha_{1\,j}t_1 + ... + \alpha_{k\,j}t_k) 
\frac{e^{i t_1\cdot ... \cdot t_k} }{t_1\cdot  ... \cdot t_k} d t_1 ... d t_k.
\end{equation}
By splitting $e^{i t_1 \cdot ...\cdot t_k}$ as

$$e^{i t_1 \cdot ...\cdot t_k} = \cos (t_1 \cdot ... \cdot t_k) + i \sin (t_1\cdot ... \cdot t_k)$$
and ignoring the harmless factor $e^{i (\# + \#_1 + ... + \#_{n-1})x^k}$,  one can decompose the rest of \eqref{2*} as

\begin{equation}\label{3*a}
\frac{(-1)^k}{(i \pi)^k} \int_{\R^k}  \widetilde{\chi}_N  (x+t_1 + ... + t_k)
\prod_{j=1}^{n-1}  \widetilde{\chi}_N   (x + \alpha_{1\,j}t_1 + ... + \alpha_{k\,j}t_k) 
\frac{\cos ( t_1\cdot ... \cdot t_k) }{t_1\cdot  ... \cdot t_k} d t_1 ... d t_k
\end{equation}

$$ + $$

\begin{equation}\label{3*b}
i \, \frac{(-1)^k}{(i \pi)^k}  \int_{\R^k}  \widetilde{\chi}_N  (x+t_1 + ... + t_k)
\prod_{j=1}^{n-1}  \widetilde{\chi}_N   (x + \alpha_{1\,j}t_1 + ... + \alpha_{k\,j}t_k) 
\frac{\sin ( t_1\cdot ... \cdot t_k) }{t_1\cdot  ... \cdot t_k} d t_1 ... d t_k.
\end{equation}

It is a good moment now to pause and make a few important remarks regarding the supports of our integrands in \eqref{3*a} and \eqref{3*b}. Consider a generic set of the form

\begin{equation}\label{set}
\{ (t_1, ..., t_k) \in \R^k : a \leq \beta_1 t_1 + ... + \beta_k t_k \leq b \}
\end{equation}
with $a < 0 < b$ and $\vec{\beta} = (\beta_1, ..., \beta_k)$ arbitrary. Clearly, this set is a {\it $k$ dimensional strip}, containing the origin and lying between the hyperspaces
$\beta_1 t_1 + ... + \beta_k t_k = a$ and $\beta_1 t_1 + ... + \beta_k t_k = b$ which are both perpendicular to the given vector $\vec{\beta}$. Moreover, the {\it width} of this
strip is $O (b-a)$.

As a consequence, the support of the integrands in \eqref{3*a} and \eqref{3*b} lies at the intersection of $n$ such $k$ dimensional strips. Since the vectors 
$\vec{\alpha_1}, ..., \vec{\alpha_k}$ are {\it generic}, if one picks $\widetilde{c}$ small enough and $x\in [-\widetilde{c} N, \widetilde{c} N]$ this intersection will be a bounded domain in $\R^k$ containing the
origin and also contained in a large cube of sidelength $O(N)$. Hence, the term \eqref{3*b} is well defined and this means that the term \eqref{3*a} is well defined as well (being the difference of two
well defined expressions).

In particular, from \eqref{2*}, \eqref{3*a} and \eqref{3*b} one can see that

$$
| n HT_{\vec{\alpha_1}, ..., \vec{\alpha_k} }    (f, f_1, ..., f_{n-1})(x) |
$$

\begin{equation}\label{4*}
\geq \frac{1}{\pi^k}
| \int_{\R^k}  \widetilde{\chi}_N  (x+t_1 + ... + t_k)
\prod_{j=1}^{n-1}  \widetilde{\chi}_N   (x + \alpha_{1\,j}t_1 + ... + \alpha_{k\,j}t_k) 
\frac{\sin ( t_1\cdot ... \cdot t_k) }{t_1\cdot  ... \cdot t_k} d t_1 ... d t_k         |
\end{equation}
for every $x\in [-\widetilde{c} N, \widetilde{c} N]$.

Next, we would like to observe that modulo some harmless {\it error terms}, one can replace all the $\widetilde{\chi}_N$ functions in \eqref{4*} by the
corresponding $\chi_N$. To see this, let us denote the $n$ linear inner expression in \eqref{4*} by $\E(\widetilde{\chi}_N, \widetilde{\chi}_N, ..., \widetilde{\chi}_N)(x)$.
One can write

$$\E(\widetilde{\chi}_N, \widetilde{\chi}_N, ..., \widetilde{\chi}_N)(x)$$

$$ = \E(\chi_N, \widetilde{\chi}_N, ..., \widetilde{\chi}_N)(x) + \E(\widetilde{\chi}_N - \chi_N, \widetilde{\chi}_N, ..., \widetilde{\chi}_N)(x)$$
and it is not difficult to see that the absolute value of the error term $\E(\widetilde{\chi}_N - \chi_N, \widetilde{\chi}_N, ..., \widetilde{\chi}_N)(x)$ is at most $O(1)$, as a consequence
of the fact that the function $\frac{\sin x}{x}$ is bounded and that $\widetilde{\chi}_N - \chi_N$ is supported on a union of two strips of width $O(1/N^{k-1})$. Iterating this argument $n$ times
one obtains from \eqref{4*} that 

$$
| n HT_{\vec{\alpha_1}, ..., \vec{\alpha_k} }    (f, f_1, ..., f_{n-1})(x) |
$$

\begin{equation}\label{5*}
\geq \frac{1}{\pi^k}
| \int_{\R^k}  \chi_N  (x+t_1 + ... + t_k)
\prod_{j=1}^{n-1}  \chi_N   (x + \alpha_{1\,j}t_1 + ... + \alpha_{k\,j}t_k) 
\frac{\sin ( t_1\cdot ... \cdot t_k) }{t_1\cdot  ... \cdot t_k} d t_1 ... d t_k         | - O(1) 
\end{equation}
for every $x\in [-\widetilde{c} N, \widetilde{c} N]$.

Clearly, the inner term on the right hand side of \eqref{5*} can be written as

\begin{equation}\label{6*}
\int_{D_x} \frac{\sin ( t_1\cdot ... \cdot t_k) }{t_1\cdot  ... \cdot t_k} d t_1 ... d t_k
\end{equation}
where $D_x$ is a compact and convex domain in $\R^k$ containing the origin and having also the property that

$$[-c_1 N, c_1 N]^k \subseteq D_x \subseteq [-C_1 N, C_1 N]^k$$
where $c_1$ is small and $C_1$ is large and they depend on $\vec{\alpha_1}, ..., \vec{\alpha_k}$ but are otherwise independent on $x\in [-\widetilde{c} N, \widetilde{c} N]$.

Split now \eqref{6*} as

$$
\int_{D_x} \frac{\sin ( t_1\cdot ... \cdot t_k) }{t_1\cdot  ... \cdot t_k} d t_1 ... d t_k
$$

\begin{equation}\label{7*}
= 
\int_{[-c_1 N, c_1 N]^k  } \frac{\sin ( t_1\cdot ... \cdot t_k) }{t_1\cdot  ... \cdot t_k} d t_1 ... d t_k +
\int_{D_x \setminus [-c_1 N, c_1 N]^k } \frac{\sin ( t_1\cdot ... \cdot t_k) }{t_1\cdot  ... \cdot t_k} d t_1 ... d t_k.
\end{equation}
We will prove in the next two sections that the the first term in \eqref{7*} is positive and is {\it bounded from below} by $c (\log N)^{k-1}$ while the absolute value of the second term in \eqref{7*} is {\it bounded from above}
by $C (\log N )^{k-2}$.

Combining these two facts with the previous \eqref{5*} will imply the desired \eqref{1*}.

\section{Lower logarithmic bounds}

In this section we prove the lower logarithmical bounds for the first term in \eqref{7*} that have been mentioned at the end of the previous Section \ref{s3}.

\begin{proposition}\label{below}
For any integer $k\geq 1$ there exists a constant $c ( = c(k))$ with the property that

\begin{equation}\label{8*}
\int_{-N}^N ... \int_{-N}^N  \frac{\sin ( t_1\cdot ... \cdot t_k) }{t_1\cdot  ... \cdot t_k} d t_1 ... d t_k \geq c (\log N)^{k-1}
\end{equation}
as long as $N$ is large enough.
\end{proposition}
The proof of this Proposition \ref{below} is based on the following lemma whose enuntiation requires some additional notations.

If $f$ is a bounded measurable function defined on the interval $[0, \infty)$ we denote by $H f(x)$ the linear operator given by

$$ H f(x) = \frac{1}{x} \int_0^x f (u) d u$$
for every $x\in [0, \infty)$. We will also denote by $H^l$ the composition of $H$ with itself $l$ times (as long as $l$ is an integer greater or equal than $1$) and by $H^0$ the identity operator.

\begin{lemma}\label{belowl}
For any integer $k\geq 1$ there exist a small constant $c_{k-1}$ and a large one $\C_{k-1}$ with the property that

\begin{equation}\label{9*}
\int_0^t H^{k-1} F (u) d u \geq c_{k-1} (\log t )^{k-1}
\end{equation}
for every $t\geq \C_{k-1}$, where $F(u) : = \frac{\sin u}{u}$.
\end{lemma}
Let us first assume this Lemma \ref{belowl} and show how our previous Proposition \ref{below} can be reduced to it.

If $a$ is any real number different than zero, the function $s\rightarrow \frac{\sin a s}{s}$ is an even function and in particular this implies that

$$\int_{-N}^N \frac{\sin a s}{s} d s = 2 \int_0^N \frac{\sin a s}{s} d s .$$
Using this observation several times, one can see that

\begin{equation}\label{10*}
\int_{-N}^N ... \int_{-N}^N  \frac{\sin ( t_1\cdot ... \cdot t_k) }{t_1\cdot  ... \cdot t_k} d t_1 ... d t_k = 2^k 
\int_{0}^N ... \int_{0}^N  \frac{\sin ( t_1\cdot ... \cdot t_k) }{t_1\cdot  ... \cdot t_k} d t_1 ... d t_k .
\end{equation}
We claim now that the following identity holds

\begin{equation}\label{11*}
\int_{0}^N ... \int_{0}^N  \frac{\sin ( t_1\cdot ... \cdot t_k) }{t_1\cdot  ... \cdot t_k} d t_1 ... d t_k  =
\int_0^{N^k} H^{k-1} F (u) d u .
\end{equation}
Clearly, if we take this equality for granted then \eqref{11*}, \eqref{10*} and \eqref{9*} together imply the desired \eqref{8*}.

For simplicity, we will prove \eqref{11*} in the particular case $k=3$ and leave the general case to the reader, since it does not require any additional ideas.

One can write

$$\int_0^N \int_0^N \int_0^N
\frac{\sin ( t_1 t_2 t_3) }{t_1 t_2 t_3} d t_1 d t_2 d t_3 =
\int_0^N \int_0^N (\int_0^N \frac{\sin ( t_1 t_2 t_3) }{t_3} d t_3 ) \frac{d t_1}{t_1} \frac{d t_2}{t_2}
$$

$$=
\int_0^N \int_0^N (\int_0^{N t_1 t_2} \frac{\sin x}{ x} d x ) \frac{d t_2}{t_2} \frac{d t_1}{t_1}
= \int_0^N \int_0^{N^2 t_1} (\int_0^y \frac{\sin x}{ x} d x ) \frac{d y}{y} \frac{d t_1}{t_1}$$

$$= \int_0^{N^3} \int_0^z ( \int_0^y \frac{\sin x}{ x} d x ) \frac{d y}{y} \frac{d z}{z} =
\int_0^{N^3} (\frac{1}{z} \int_0^z (\frac{1}{y} \int_0^y \frac{\sin x}{ x} d x ) d y) d z$$

$$= \int_0^{N^3} H^2 F (z) d z$$
as desired.

We are left with the proof of Lemma \ref{belowl}. We proceed by induction. Clearly, the $k=1$ case is a simple
consequence of the fact that $\int_0^{\infty} \frac{\sin x}{x} d x = \frac{\pi}{2}$. Suppose now that \eqref{9*} holds for the parameter $k-1$
and we would like to prove it for $k$.

One writes

\begin{equation}\label{12*}
\int_0^t H^k F (u) d u = \int_0^t ( \frac{1}{x} \int_0^x H^{k-1} F (u) d u ) d x 
\end{equation}

\begin{equation}\label{13*}
= \int_0^{\C_{k-1}} ... + \int_{\C_{k-1}}^t ... \, .
\end{equation}
The absolute value of the first term in \eqref{13*} is clearly at most $\C_{k-1}$ given that $|F(u)| \leq 1$. Using
the induction hypothesis on the other hand, one can estimate the second term in \eqref{13*} from below by

$$c_{k-1} \int_{\C_{k-1}}^t \frac{1}{x} (\log x )^{k-1} d x = \frac{c_{k-1}}{k} \int_{\C_{k-1}}^t [ (\log x )^k ]' d x
$$

$$ = \frac{c_{k-1}}{k} ( (\log t )^k - (\log \C_{k-1})^k ) .$$
All of these imply that the left hand side of \eqref{12*} can be estimated from below by 

$$ = \frac{c_{k-1}}{k} ( (\log t )^k - (\log \C_{k-1})^k ) - \C_{k-1} .$$

But this expression is at least as big as $\frac{c_{k-1}}{2 k} (\log t )^k$
if $t$ is large enough and this completes the proof of Lemma \ref{belowl} and therefore of Proposition \ref{below}.

\section{Upper logarithmic bounds}

Our final goal now is to prove the upper logarithmical bounds for the second term in \eqref{7*}, that have been claimed in Section \ref{s3}.

\begin{proposition}\label{p2}
Let $D$ be a compact and convex domain in $\R^k$ having the property that 

$$[-c N, c N]^k \subseteq D \subseteq [-C N , C N]^k$$
where $c$ and $C$ are fixed constants and $N$ is large enough. Then, there exists $\widetilde{C} ( = \widetilde{C}(k))$ so that

\begin{equation}\label{14*}
| \int_{D \setminus [-c N, c N]^k} 
\frac{\sin ( t_1\cdot ... \cdot t_k) }{t_1\cdot  ... \cdot t_k} d t_1 ... d t_k |
\leq \widetilde{C} ( \log N )^{k-2}.
\end{equation}

\end{proposition}
We claim that the above Proposition \ref{p2} follows easily from the following

\begin{lemma}\label{l2}
Let $D$ be a compact and convex domain in $\R^k$ having the property that

$$D \subseteq [ -N, N]^k.$$
Then, there exists $C ( = C(k))$ so that

\begin{equation}\label{15*}
| \int_D \frac{\sin ( t_1\cdot ... \cdot t_k) }{t_1\cdot  ... \cdot t_k} d t_1 ... d t_k | \leq C (\log N )^{k-1}.
\end{equation}

\end{lemma}
Let us see first why Lemma \ref{l2} implies Proposition \ref{p2}.

If $\vec{t} = (t_1, ..., t_k) \in D \setminus [-c N, c N]^k$ then at least for one index $1\leq i\leq k$ one must have $t_i \notin [-c N, c N]$. But then, this means that
$t_i \in [-C N, - c N] \cup [c N , C N]$. 

Let us examine now the following two {\it extremal} cases.

Suppose first that $t_i \in [-C N, - c N] \cup [c N , C N]$ for {\it every} $1\leq i\leq k$. In this case it is not difficult to see that the integral over that corresponding
region is at most

$$\int_{c N}^{C N} ... \int_{c N}^{C N}
\frac{ d t_1}{t_1} ... 
\frac{ d t_k}{t_k}
$$
which is clearly bounded by a constant independent of $N$.

Assume now that we are in the opposite situation when precisely one index $i$ has the property that $t_i \in [-C N, - c N] \cup [c N , C N]$. By symmetry, we can also assume
that that index is $1$ and that $t_1 \in [ c N , C N]$. In this case, it is also not difficult to see that the integral over the corresponding region can be expressed as

\begin{equation}\label{16*}
\int_{c N}^{C N} \frac{1}{t_1}
(\int_{D_{t_1}} \frac{\sin ( t_1\cdot ... \cdot t_k) }{t_2\cdot  ... \cdot t_k} d t_2 ... d t_k ) d t_1
\end{equation}
where

$$D_{t_1} : = \{ (t_2, ..., t_k) : (t_1, t_2, ..., t_k) \in D \}.
$$
It is natural to change variables $t_1^{1/k-1} t_j = s_j$ for $2\leq j\leq k$ and rewrite \eqref{16*} as 

\begin{equation}\label{17*}
\int_{c N}^{C N} \frac{1}{t_1}
(\int_{\widetilde{D}_{t_1}} \frac{\sin ( s_2\cdot ... \cdot s_k) }{s_2\cdot  ... \cdot s_k} d s_2 ... d s_k ) d t_1
\end{equation}
where $\widetilde{D}_{t_1}$ is also compact and convex and has the property that

$$\widetilde{D}_{t_1} \subseteq [ - C N t_1^{1/k-1}, C N t_1^{1/k-1} ] ^{k-1}.
$$
Using Lemma \ref{l2} one can then bound \eqref{17*} easily by $C ( \log N ) ^{k-2}$ which is of course acceptable by \eqref{14*}.

The general case when an arbitrary number of indices $i$ satisfy $t_i \in [-C N, - c N] \cup [c N , C N]$ can be treated similarly and the corresponding upper bound will be of the form 
$C (\log N) ^l$ for some $0\leq l \leq k-2$. Since there are only a finite number of such situations, this completes the proof of \eqref{14*}.

We are left with the proof of Lemma \ref{l2}. We proceed as before by induction.

The case $k=1$ is obviously true, since $D$ is now an interval $[a , b]$ and \eqref{15*} becomes equivalent to

$$| \int_a^b \frac{\sin x}{x} d x | \leq C.$$
Let us consider now the general case of \eqref{15*} assuming (by the induction hypothesis) that all the previous ones are known.

Decompose the inner integral in \eqref{15*} as

\begin{equation}\label{18*}
\int_D \frac{\sin ( t_1\cdot ... \cdot t_k) }{t_1\cdot  ... \cdot t_k} d t_1 ... d t_k
\end{equation}

$$ = \int_{ D \cap \{ \vec{t} : |\vec{t}|_{\infty}\leq 1 \}} \frac{\sin ( t_1\cdot ... \cdot t_k) }{t_1\cdot  ... \cdot t_k} d t_1 ... d t_k
+  \sum_{d=0}^{\log N} \int_{ D \cap \{ \vec{t} : 2^d < |\vec{t}|_{\infty}\leq 2^{d+1} \}} \frac{\sin ( t_1\cdot ... \cdot t_k) }{t_1\cdot  ... \cdot t_k} d t_1 ... d t_k 
\footnote{ We use the notation $|\vec{t}|_{\infty} := \max_{1\leq i\leq k} |t_i|$.} .
$$
Arguing as before and using the induction hypothesis one can see that

\begin{equation}\label{19*}
| \int_{ D \cap \{ \vec{t} : 2^d < |\vec{t}|_{\infty}\leq 2^{d+1} \}} \frac{\sin ( t_1\cdot ... \cdot t_k) }{t_1\cdot  ... \cdot t_k} d t_1 ... d t_k | 
\leq C d^{k-2}.
\end{equation}

Finally, using \eqref{19*} in \eqref{18*} one obtains the desired \eqref{15*}.

\section{Further remarks}

First of all, as we promised, we would like to explain why Proposition \ref{precis} holds true even for $k=2$ and $n=4$. Recall from \eqref{kernel} the kernel representation
of $4 HT_{\vec{\alpha_1}, \vec{\alpha_2}} (f, f_1, f_2, f_3)(x)$ as

\begin{equation}\label{[1]}
-\frac{1}{\pi^2} p. v. \int_{\R^2}
f(x+t+s)
\prod_{j=1}^3
f_j(x+\alpha_{1\, j}s + \alpha_{2\, j}t)
\frac{d t}{t} \frac{d s}{s}
\end{equation}
where $\vec{\alpha_1} = (\alpha_{1\,1}, \alpha_{1\, 2}, \alpha_{1\,3})$ and $\vec{\alpha_2} = (\alpha_{2\,1}, \alpha_{2\, 2}, \alpha_{2\,3})$ are two generic vectors in $\R^3$.

Consider as before $f(x) = e^{ i \# x^2}$ and $f_j(x) = e^{ i \#_j x^2}$ for $j=1,2,3$ where $\#, \#_1, \#_2, \#_3$ are real numbers that will be determined later on. If one formally plugs in these
functions into \eqref{[1]}, the corresponding expression in \eqref{polinom} becomes a polynomial in the variables $x, t, s$ which is homogeneous of degree $2$. This polynomial has
precisely six monomials, namely $xt, xs, ts, t^2, s^2$ and $x^2$ each of which having its corresponding coefficient.
We would like to choose our numbers $\#, \#_1, \#_2, \#_3$ so that the coefficients of $xt, xs$ and $s^2$ are all zero. As we discussed earlier in Section \ref{s2}, this amounts
to pick a vector $(\#, \#_1, \#_2, \#_3) \in \R^4$ orthogonal to three other generic linearly independent $4$ dimensional vectors, which is clearly possible. Using this choice, the analogous of 
\eqref{alta} becomes

\begin{equation}\label{[2]}
| 4 HT_{\vec{\alpha_1}, \vec{\alpha_2}} (f, f_1, f_2, f_3)(x) | = \frac{1}{\pi^2}
| \int_{\R^2} e^{i\alpha t^2} e^{i\beta t s} \frac{d t}{t} \frac{d s}{s} | 
\end{equation}
where $\alpha, \beta$ are real numbers depending on the previous parameters $\vec{\alpha_1}, \vec{\alpha_2}$ and $\#, \#_1, \#_2, \#_3$. By construction, one can
also assume without loss of generality that $\alpha > 0$. As in Section \ref{s2} one then observes that the expression on the right hand side of \eqref{[2]} can be calculated further as

$$
| \int_{\R^2} e^{i\alpha t^2} e^{i\beta t s} \frac{d t}{t} \frac{d s}{s} |  = \frac{1}{\pi} | \int_{\R} \frac{\sgn (t)}{t} e^{i\alpha t^2} d t |
$$

$$ = \frac{1}{\pi} | \int_0^{\infty} \frac{ e^{i\alpha t}}{t} d t | = \frac{1}{\pi} | \int_0^{\infty} \frac {\cos t}{t} d t + i \frac{\pi}{2} |
$$
and while $\int_1^{\infty} \frac {\cos t}{t} d t$ is bounded, the integral $\int_0^1 \frac {\cos t}{t} d t$ is infinite.

To transform this heuristical argument into a rigorous one, one proceeds as before. The details are left to the reader.

Then, we would like to describe the construction of the {\it reducible} counterexamples from the previous {\it irreducible} ones, completing in this way the proof of Theorem \ref{teorema}. 
Fix $k\geq 2$ and $n$ such that $n\geq N(k)$ (here if $k=2$ we should replace $N(2)$ by $4$ since we now know
that Proposition \ref{precis} still holds in this situation). Consider also a generic operator of type $n HT_{\vec{\alpha_1}, ..., \vec{\alpha_{k'}} }$ where $k'\geq k$. To construct such a {\it reducible} counterexample for it 
we proceed as follows. At the first step, take the {\it irreducible} counterexample given by Proposition \ref{precis} for the less complex operator $n HT_{\vec{\alpha_1}, ..., \vec{\alpha_{k}} }$. By a simple approximation argument
one can also assume without loss of generality that the functions which appear in the counterexample are all compactly suppported in frequency. Then, by using the dilation invariance of these operators,
one can rescale it, and obtain a counterexample whose Fourier transform is supported inside the unit cube of $\R^n$. After that, using the modulation invariance of the operators, one observes that {\it this unit cube}
can be translated anywhere along the subspace $$\Gamma_{\vec{\alpha_1}}\cap ... \cap \Gamma_{\vec{\alpha_k}}$$ and still remains a counterexample for the boundedness of $n HT_{\vec{\alpha_1}, ..., \vec{\alpha_{k}} }$. It is not difficult to observe that 
one can do this in such a way that the new translated unit cube does not intersect any of the remaining subspaces $\Gamma_{\vec{\alpha_{k+1}}}, ... , \Gamma_{\vec{\alpha_{k'}}}$. But then this means that
these new functions which correspond to the new rescaled and translated cube, automatically become a counterexample for our original, more complex operator $n HT_{\vec{\alpha_1}, ..., \vec{\alpha_{k'}} }$.

It is also natural to ask, given the previous Theorem \ref{general} and Theorem \ref{mult}, what can be said in the remaining cases, when the dimension of the singularity subspace $\Gamma$ satisfies

$$\dim(\Gamma) \geq \frac{n+1}{2} = n - \frac{n-1}{2}.$$
Using an argument similar to the one before, one can prove

\begin{corollary}\label{coro}
Let $n\geq 5$. For any $1<p_1, ..., p_n \leq \infty$ and $0<p<\infty$ with $1/p_1+ ... +1/p_n = 1/p$,
and for any integer $k$ satisfying

$$k > n - (\frac{n-1}{2})^{1/2}$$
there exist non-degenerate subspaces $\Gamma \subseteq \R^n$ with $\dim(\Gamma) = k$
and symbols $m \in \M_{\Gamma}(\R^n)$, for
which the associated  $n$-linear multiplier operators $T_m$ do not map $L^{p_1}\times ... \times L^{p_n}$ into $L^p$.
\end{corollary}

\begin{proof}
Let $d\geq 1$ and denote by $K(\vec{t}) = \frac{t_1}{|\vec{t}|^{d+1}}$ the first $d$-dimensional Riesz kernel. It is a classical well known fact that $\widehat{K}(\vec{\eta}) = C_d \frac{\eta_1}{|\vec{\eta}|}$
where $C_d$ is a constant depending only on the dimension, see for instance \cite{sw}. For $n\geq d$ consider the $n$-linear operator defined by the formula

\begin{equation}\label{riesz}
\int_{\R^{2d}} f_1 (x - \vec{a_1}\cdot \vec{t} - \vec{b_1}\cdot \vec{s})\cdot ... \cdot  f_n (x - \vec{a_n}\cdot \vec{t} - \vec{b_n}\cdot \vec{s})         K(\vec{t}) K(\vec{s}) \, d \vec{t} \,d \vec{s}
\end{equation}
where $\vec{a_j}$, $\vec{b_j}$ are generic vectors in $\R^d$ while $\vec{a_j}\cdot \vec{t}$ and $\vec{b_j}\cdot \vec{s}$ are $d$ dimensional inner products, for $1\leq j\leq d$. 
Recall that the functions $f_1$, ..., $f_n$ are all defined on the real line.

Alternatively, as before, one can rewrite \eqref{riesz} as 

$$
\int_{\R^n} \widehat{K}( A \vec{\eta})  \widehat{K}( B \vec{\eta})
\widehat{f_1}(\eta_1)  ... \widehat{f_n}(\eta_n) e^{2 \pi i x (\eta_1 + ... + \eta_n)} d\eta_1 ... d\eta_n
$$
where $A = [\vec{a_1} ... \vec{a_n}]$ and $B = [\vec{b_1} ... \vec{b_n}]$ are both matrices having $d$ lines and $n$ columns and they define linear maps from $\R^n$ to $\R^d$. Since 
$\widehat{K}$ is singular only at the origin, it is clear that the symbol of our operator $\vec{\eta} \rightarrow \widehat{K}( A \vec{\eta})  \widehat{K}( B \vec{\eta})$ will be singular along
$\ker (A) \cup \ker (B)$ and for generic matrices $A$ and $B$ both of these subspaces will have dimension $n-d$. Consider now $f_j(x) = e^{2\pi i \#_j x^2}$ for $1\leq j\leq n$ where the real numbers
$(\#_j)_{j=1}^n$ will be determined later. If one formally plugs in these functions into the formula \eqref{riesz}, the exponent of the new complex exponential is the quadratic expression

\begin{equation}\label{quad}
\#_1 (x - \vec{a_1}\cdot \vec{t} - \vec{b_1}\cdot \vec{s})^2 + ... + \#_n (x - \vec{a_n}\cdot \vec{t} - \vec{b_n}\cdot \vec{s})^2
\end{equation}
depending on the variables $x$, $t_1$, ..., $t_d$, $s_1$, ..., $s_d$. If one expands \eqref{quad}, one can see by an elementary calculation, that it contains $2d^2+2d+1$ quadratic monomials.
We choose now the vector $\vec{\#} = (\#_1, ..., \#_n)$ in such a way that all the coefficients of these monomials vanish, with the exception of the coefficients corresponding to $x^2$, $t_1s_1$, ..., $t_ds_d$.
As we have seen before, this is equivalent to the fact that $\vec{\#}$ is orthogonal to $2d^2+d$ other linearly independent $n$ dimensional vectors, which is clearly possible as long as $n\geq 2d^2+d+1$.
From this we deduce that $d < (\frac{n-1}{2})^{1/2}$ which in particular implies 

\begin{equation}\label{ker}
\dim (\ker (A)) = \dim (\ker (B)) >  n - (\frac{n-1}{2})^{1/2}.
\end{equation}
On the other hand, it is not difficult to see that the absolute value of the corresponding expression in \eqref{riesz}, is comparable to the divergent integral

$$\int_{\R^d}\frac{t_1^2}{|\vec{t}|^{d+2}} d \vec{t}$$
a fact that can be used, also as before, to show that the $n$-linear operator \eqref{riesz} does not satisfy {\it any} $L^p$ estimates. Since this operator can be also naturally decomposed as
$T_{m_1} + T_{m_2}$ with $m_1\in \M_{\ker (A)}(\R^n)$ and $m_1\in \M_{\ker (B)}(\R^n)$, it is clearly impossible for both $T_{m_1}$ and $T_{m_2}$ to satisfy the required particular estimates of Corollary \ref{coro}.
\end{proof}

Another interesting consequence is the following. On the one hand, let us observe that Theorem \ref{teorema} holds not only for the operators $n HT_{\vec{\alpha_1}, ..., \vec{\alpha_k} }$,
but also for operators whose symbols are given by products of type \eqref{psymbol}, where ``$\sgn$'' is replaced by ``$1_{\R_+}$'' the characteristic function of the set of positive real numbers. Indeed, this change
corresponds to replacing the previous kernels ``$\frac{1}{t}$'' by ``$\frac{1}{t} + \delta_0(t)$'' (where $\delta_0(t)$ is the Dirac delta distribution centered at the origin) and it is not difficult to see that this does
not change the outcome of the previous arguments. In particular, when $k=2$, this implies that $n$-linear operators with symbols such as the ones described on the right hand side of Figure \ref{fig1}, do not satisfy 
any $L^p$ estimates of H\"{o}lder type \footnote{The diagram should be understood in $\R^n$. The two lines represent two generic hyperspaces $\Gamma_1$ and $\Gamma_2$.}.

\begin{figure}[htbp]\centering
\psfig{figure=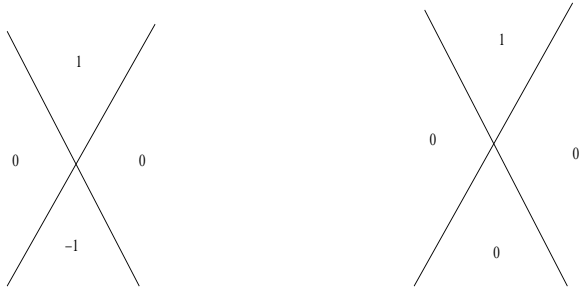, height=1.5in, width=3in}
\caption{Good (?) and bad symbols}
\label{fig1}
\end{figure}
On the other hand, if the $n$-linear Hilbert transforms were to satisfy some $L^p$ estimates, then by taking the difference between two generic ones, one would obtain the same $L^p$ estimates for
 $n$-linear operators given by symbols such as the ones described on the left hand (this time) side of Figure \ref{fig1}.

Let us end with a few remarks on the previous identity \eqref{2}. First of all, it can be suggestively rewritten as 

\begin{equation}\label{[3]}
\int_{[0,1]^{n-1}}
n HT_{\vec{\alpha}}(f, a, ..., a)(x) d\vec{\alpha} = p. v. \int_{\R}
\left( \frac{\Delta_t}{t} A(x) \right)^{n-1} f(x+t) \frac{d t}{t}
\end{equation}
where in general $\Delta_t g(x) : = g(x+t) - g(x)$ is the finite difference of the function $g$ at scale $t$. In \cite{camil2} the following generalization of it has been noticed

$$
\int_{[0,1]^{n-1}} ... \int_{[0,1]^{n-1}} n HT_{\vec{\alpha_1}, ..., \vec{\alpha_k} }    (f, a, ..., a)(x) d\vec{\alpha_1} ... d\vec{\alpha_k} 
$$

\begin{equation}\label{[4]}
= p. v. \int_{\R^k}
\left( \frac{\Delta_{t_1}}{t_1}\circ ... \circ \frac{\Delta_{t_k}}{t_k} A(x) \right)^{n-1} f( x+t_1 + ... + t_k) \frac{d t_1}{t_1} ... \frac{d t_k}{t_k}
\end{equation}
where this time $A^{(k)} = a$. It is interesting to mention that the linear operators on the right hand side of \eqref{[4]} {\it are} bounded on $L^p$ for every $1<p<\infty$ as long as 
$A^{(k)} \in L^{\infty}$. 

These operators appeared naturally in \cite{camil2} as part of a generalization of Calder\'{o}n's theory to classes of functions having arbitrary
{\it polynomial growth}. For more details, the reader is referred to the recent sequel of the author \cite{camil3}, \cite{camil2} and \cite{camil4} .

\end{document}